\documentclass[12pt]{amsart}

\usepackage[english]{babel}

\usepackage[pdftex,textwidth=400pt,marginratio=1:1]{geometry}

\usepackage{amsfonts}

\usepackage[dvips]{graphics}

\usepackage{amsmath}

\usepackage{amsthm}

\usepackage{amssymb}

\usepackage{eufrak}

\usepackage{cancel}

\usepackage{color}

\usepackage[curve]{xypic}

\newtheorem{theorem}{Theorem}[section]

\newtheorem{corollary}[theorem]{Corollary}

\newtheorem{proposition}[theorem]{Proposition}

\newtheorem{lemma}[theorem]{Lemma}

\theoremstyle{definition}

\theoremstyle{property}

\makeatletter
\newcommand{\contraction}[5][1ex]{%
  \mathchoice
    {\contraction@\displaystyle{#2}{#3}{#4}{#5}{#1}}%
    {\contraction@\textstyle{#2}{#3}{#4}{#5}{#1}}%
    {\contraction@\scriptstyle{#2}{#3}{#4}{#5}{#1}}%
    {\contraction@\scriptscriptstyle{#2}{#3}{#4}{#5}{#1}}}%
\newcommand{\contraction@}[6]{%
  \setbox0=\hbox{$#1#2$}%
  \setbox2=\hbox{$#1#3$}%
  \setbox4=\hbox{$#1#4$}%
  \setbox6=\hbox{$#1#5$}%
  \dimen0=\wd2%
  \advance\dimen0 by \wd6%
  \divide\dimen0 by 2%
  \advance\dimen0 by \wd4%
  \vbox{%
    \hbox to 0pt{%
      \kern \wd0%
      \kern 0.5\wd2%
      \contraction@@{\dimen0}{#6}%
      \hss}%
    \vskip 0.5ex%  how far above the line starts
    \vskip\ht2}}

\newcommand{\contraction@@}[3][0.05em]{%
  \hbox{%
    \vrule width #1 height 0pt depth #3%
    \vrule width #2 height 0pt depth #1%
    \vrule width #1 height 0pt depth #3%
    \relax}}
\makeatother

\DeclareFontFamily{OT1}{rsfs}{}
\DeclareFontShape{OT1}{rsfs}{n}{it}{<-> rsfs10}{}
\DeclareMathAlphabet{\curly}{OT1}{rsfs}{n}{it}

\newcommand\I{\curly I}

\renewcommand\L{\mathcal L}

\renewcommand\O{\mathcal O}
\newcommand\PP{\mathbb P}

\newcommand\cP{\curly P}
\newcommand\Pb{\mathcal P_\beta}

\newcommand\D{\mathcal D}

\newcommand\C{\mathbb C}

\newcommand\II{\mathbb I}

\newcommand\Pg{\mathcal P_\gamma}

\newcommand\R{\mathbb R}
\newcommand\Z{\mathbb Z}

\newcommand{\Rt}[1]{\stackrel{#1\,}{\longrightarrow}}

\newcommand\into{\hookrightarrow}
\newcommand\Into{\ar@{^(->}[r]<-.3ex>}

\newcommand\bull{{\scriptscriptstyle\bullet}}
\newcommand\udot{^\bull}

\newcommand\rk{\operatorname{rank}}

\newcommand\Td{\operatorname{Td}}

\renewcommand\div{\operatorname{div}}
\newcommand\id{\operatorname{id}}

\renewcommand\hom{\curly H\!om}

\newcommand\Pic{\operatorname{Pic}}

\newcommand\Hilb{\operatorname{Hilb}}

\newcommand\beq[1]{\begin{equation}\label{#1}}
\newcommand\eeq{\end{equation}}
\newcommand\beqa{\begin{eqnarray*}}
\newcommand\eeqa{\end{eqnarray*}}

\DeclareRobustCommand{\SkipTocEntry}[4]{}
\begin{document}
\title[Curve counting on surfaces II: calculations]{Reduced classes and curve counting on surfaces II: calculations}
\author[M. Kool and R. P. Thomas]{Martijn Kool and Richard Thomas \vspace{-5mm}}
\maketitle

\begin{abstract}
We calculate the stable pair theory of a projective surface $S$. For fixed curve class $\beta\in H^2(S)$ the results are entirely topological, depending on $\beta^2$, $\beta.c_1(S)$, $c_1(S)^2$, $c_2(S)$, $b_1(S)$ \emph{and} invariants of the ring structure on $H^*(S)$ such as the Pfaffian of $\beta$ considered as an element of $\Lambda^2 H^1(S)^*$. Amongst other things, this proves an extension of the G\"ottsche conjecture to non-ample linear systems.

We also give conditions under which this calculates the full 3-fold reduced residue theory of $K_S$. This is related to the reduced residue Gromov-Witten theory of $S$ via the MNOP conjecture.
When the surface has no holomorphic 2-forms this can be expressed as saying that certain Gromov-Witten invariants of $S$ are topological.

Our method uses the results of \cite{KT1} to express the reduced virtual cycle in terms of Euler classes of bundles over a natural smooth ambient space.
\end{abstract}
\thispagestyle{empty}
\renewcommand\contentsname{\vspace{-8mm}}
\tableofcontents

\section{Introduction}

Fix a nonsingular projective surface $S$ and a homology class $\beta \in H_{2}(S,\Z)$ of Hodge type $(1,1)$. Whenever there exists a deformation of $S$ for which $\beta$ is no longer $(1,1)$, the conventional Gromov-Witten invariants \cite{Beh, BF, LT} of $S$ and of the Calabi-Yau 3-fold $X:=K_S$ vanish by deformation invariance. Similarly for the stable pair invariants of $X$ \cite{PT1}. By removing part of the obstruction bundle one can define ``\emph{reduced}" invariants which are only invariant under deformations of $S$ in the Noether-Lefschetz locus (the locus where $\beta$ is of type $(1,1)$). Various authors have studied this in various contexts; see the introduction to \cite{KT1} for references. In \cite{KT1} we defined such reduced Gromov-Witten and stable pair invariants under the condition that
\begin{equation} \label{cup}
H^{1}(T_{S}) \stackrel{\cup \beta}{\longrightarrow} H^{2}(\mathcal{O}_{S}) \mathrm{ \ is \ surjective}.
\end{equation} 
Here we consider $\beta$ to lie in $H^{1}(\Omega_{S})$, so the map is induced by the pairing $\Omega_{S} \otimes T_{S} \to \mathcal{O}_{S}$. When $h^{2,0}(S) = 0$ the reduced invariants coincide with the ordinary invariants.

In this paper we work with the stable pair theory. As in \cite{KT1} one can then work under the weaker condition that 
\begin{equation} \label{h2=0}
H^{2}(L) = 0 \mathrm{ \ for \ all \ line \ bundles \ } L \mathrm{ \ with \ } c_{1}(L) = \beta.
\end{equation}
The natural $\mathbb{C}^{*}$-action on the fibres of $X=K_S$ lifts to an action on the moduli space of stable pairs on $X$. One of the connected components of the fixed point locus is the moduli space of stable pairs on the surface
% \footnote{Recall \cite[Appendix B]{PT3} that on a surface, stable pairs
% with $\chi=n$ are equivalent to pairs $(C,Z)$ where $C$ is a pure curve
% in class $\beta$ and $Z$ is a length-$(n-1+h)$ subscheme $Z\subset C$.
% Here $h$ is the arithmetic genus of $C$, given by the adjunction formula
% $2h-2=\beta^2-c_1(S).\beta$.}
$S$. (There can be other components containing stable pairs supported set-theoretically but not scheme-theoretically on $S$.) By $\C^*$-localisation we get a reduced obstruction theory on the moduli space of stable pairs on $S$. In \cite[Appendix A]{KT1}, written with D.~Panov, we are able to identify this reduced obstruction theory with the one that arises naturally in a completely different way. Namely, we take the zero locus of a section of a bundle over a natural smooth ambient space, then a section of another bundle over this zero locus cuts out the moduli space.

%Note that assumption (\ref{cup}) implies assumption (\ref{h2=0}) as follows. Let $C$ be an effective divisor with class $\beta$ and set $L = \mathcal{O}_{S}(C)$. Consider the short exact sequence
%\begin{equation*}
%0 \longrightarrow \mathcal{O}_{S} \longrightarrow L \longrightarrow N_{C} \longrightarrow 0
%\end{equation*} 
%where $N_{C}$ is the normal bundle of $C$ in $S$. The associated long exact sequence gives rise to a connecting morphism $H^{1}(N_{C}) \longrightarrow H^{2}(\mathcal{O}_{S})$ known as the semi-regularity map \cite{Blo}. Since the map (\ref{cup}) factors through the semi-regularity map \cite{Blo}, it will be surjective so $H^{2}(L) = 0$. 

This allows us to calculate the (reduced, residue) stable pair invariants of $S$ in terms of integrals over the smooth ambient space against the Euler classes of the two bundles.

While it is a general principle that stable pairs are easier to calculate with than stable maps, we know of almost no other moduli problem where such direct calculation is possible.\footnote{The genus zero Gromov-Witten theory of complete intersections in convex varieties is perhaps the only other case.} Usually obtaining explicit results is very complicated, involving various difficult degeneration and localisation tricks.

In forthcoming work \cite{PTKKV} this calculation provides one of the foundations of a computation of the full stable pairs theory of the twistor family of a K3 surface. Via Pandharipande and Pixton's recent proof of the MNOP conjecture for many 3-folds \cite{PaPi}, this then gives a proof of the famous KKV formula for the Gromov-Witten invariants of K3 surfaces in all genera, degrees and for all multiple covers. \bigskip

We split the calculation up into two cases.
In the first we simplify things by using $H_1$-insertions \cite{BL, KT1} to cut the moduli space down to curves living in a single linear system $|L|$, where $c_1(L)=\beta$. When $L$ is sufficiently ample the moduli space is smooth, the reduced obstruction bundle vanishes and the expressions reduce to the intersection numbers encountered in \cite{KST}. In \cite{KST} some of these intersection numbers were related to counts of nodal curves on $S$ and used to prove the G\"ottsche conjecture. Here we work with arbitrary $L$ satisfying Condition \eqref{h2=0}, where the invariants with $H_1$-insertions include an extension of G\"ottsche's invariants to the non-ample case \cite[Section~5]{KT1}. We show that, just as in the G\"ottsche case, the invariants only depend on the four topological numbers $\beta^{2}$, $\beta.c_{1}(S)$, $c_{1}(S)^{2}$, $c_{2}(S)$.

Let $h$ denote the arithmetic genus of curves in class $\beta$ as given by the adjunction formula
\beq{genus}
2h-2=\beta^2-c_1(S).\beta.
\eeq
\begin{theorem} \label{1}
Fix $\beta$ satisfying Condition \eqref{h2=0}. The reduced residue
invariant\footnote{This is the surface part of the $\C^*$-equivariant stable pair invariant of $X$. Up to the power of the equivariant cohomology parameter $t$, and saying ``virtual" at the appropriate places, it works out to be the following. Integrate the Chern class of the cotangent bundle of the moduli space of pairs over the subspace of pairs whose underlying curves live in $|L|$ and pass through $m$ fixed generic points of $S$. The $\gamma_i$ form an integral oriented basis of $H_1(S)$/torsion; their insertion cuts $H_\beta$ down to a single linear system $|L|$.\label{ftn}}
$\cP_{\!1-h+n,\beta}^{red}(S, [\gamma_{1}] \ldots [\gamma_{b_{1}(S)}]
[pt]^{m})$ $\in\Z(t)$ of \cite[Section 3.2]{KT1} is the product of
$t^{m+h^{0,1}(S)-h^{0,2}(S)}$ and a universal function of the variables
\beq{6nos}
n,\,m,\,\beta^{2},\ \beta.c_{1}(S),\ c_{1}(S)^{2},\ c_{2}(S).
\eeq
For fixed $n,m$ and $\chi(L)=\chi(\O_S)+\frac12(\beta^2+\beta.c_1(S))$ it is $(-1)^{\chi(L)-1-m+n}$ times by a polynomial in the 4 topological numbers $\beta^{2},\,\beta.c_{1}(S),\,c_{1}(S)^{2}$ and $c_{2}(S)$.
\end{theorem}

In the second case we work with stable pairs over the full Hilbert scheme $\mathrm{Hilb}_{\beta}(S)$ of curves on $S$ with class $\beta$, with no $H_1$-insertions. This differs from the first case when the dimension $h^{0,1}(S)$ of the Picard variety $\Pic_{\beta}(S)$ of $S$ is positive. This time one also has to perform integrals over $\Pic_{\beta}(S)$. The resulting invariants are again topological, depending not only on the topological numbers \eqref{6nos} but also on numerical invariants of the ring structure of $H^*(S)$ described as follows. 

Via wedging and integration over $S$ the classes $\beta,\,c_{1}(S) \in H^{2}(S,\Z)$ and $1\in H^0(S,\Z)$ give rise to elements
\beq{pfa}
[\beta],\ [c_{1}(S)] \in \Lambda^{2}H^{1}(S,\Z)^{*}\quad\mathrm{and}\quad
[1] \in \Lambda^{4} H^{1}(S,\Z)^{*}.
\eeq
Wedging together combinations of these we can get elements of $\Lambda^{\!b_{1}(S)} H^{1}(S,\Z)^{*}$:
\beq{numbs}
\Lambda^{i} [\beta] \wedge \Lambda^{j} [c_{1}(S)] \wedge \Lambda^{k} [1] \quad\mathrm{where\ \ } 2i+2j+4k = b_{1}(S).
\eeq
There is a canonical isomorphism $\Lambda^{\!b_{1}(S)} H^{1}(S,\Z)^{*} \cong \Z$ given by evaluating on the wedge of any integral basis\footnote{Note that $H^1(S,\Z)$ is torsion-free.} of $H^{1}(S,\Z)$ which is compatible in $H^1(S,\R)$ with the orientation provided by the complex structure. Therefore we can regard the $\Lambda^{i} [\beta] \wedge \Lambda^{j} [c_{1}(S)] \wedge \Lambda^{k} [1]$ as integers.

\begin{theorem} \label{2}
Fix $\beta$ satisfying \eqref{h2=0}. 
The reduced residue stable pair invariant\footnote{This invariant should be interpreted as in footnote \ref{ftn}, except the curve passing through $m$ fixed points is no longer constrained to lie in $|L|$.} $\cP_{\!1-h+n,\beta}^{red}(S, [pt]^{m})\in\Z(t)$ of \cite[Section 3.2]{KT1}
is equal to $t^{m-h^{0,2}(S)}$ times by a universal function of
\begin{equation*}
m,\,n,\,b_1(S),\,\beta^{2},\,\beta.c_{1}(S),\,c_{1}(S)^{2},\,c_{2}(S),\,\big\{\Lambda^i [\beta] \wedge \Lambda^{j} [c_{1}(S)] \wedge \Lambda^{k} [1]\big\}_{2i+2j+4k=b_1(S)}.
\end{equation*}
For fixed $n,m,b_1(S)$ and $\chi(L)=\chi(\O_S)+\frac12(\beta^2+\beta.c_1(S))$ it is the product of $(-1)^{\chi(L)-1-m+n+h^{0,1}(S)}$ and a universal polynomial in the topological numbers $\beta^{2},\,\beta.c_{1}(S),
\,c_{1}(S)^{2},\,c_{2}(S)$ and $\Lambda^i [\beta] \wedge \Lambda^{j} [c_{1}(S)] \wedge \Lambda^{k} [1]$.
\end{theorem}
 
In Section \ref{3f} we give some conditions under which the moduli space of stable pairs on $S$ is the whole fixed point locus of the moduli space of stable pairs on $X=K_S$. The most obvious case is when $\beta$ is an irreducible class. Another is when $K_{S}^{-1}$ is nef, $\beta$ is $(2\delta+1)$-very ample\footnote{By this we mean there exists a line bundle $L$ in $\Pic_\beta(S)$ which is $(2\delta+1)$-very ample. Recall \cite{BS} that this means that $H^0(L)\to H^0(L|_Z)$ is surjective for every length $2\delta+2$ subscheme $Z$ of $S$.} and the number of free points of the stable pairs is $\leq \delta$. A third example is provided by using only moduli spaces cut down by many point insertions (see also \cite[Section~5]{KT1}).

So in these cases we compute the corresponding reduced stable pair invariants of $X$, not just $S$. By the MNOP conjecture \cite{MNOP} (proved in the toric case \cite{MOOP, MPT}, the ``G\"ottsche case'' \cite{KT1}, and now for ``most" compact Calabi-Yau 3-folds \cite{PaPi}) this determines various reduced $\C^*$-equivariant GW\footnote{These are given by reduced GW invariants of $S$ with $K_S$-twisted $\lambda$-classes.} and DT invariants of $X$, which are therefore also topological. Note that in the toric case $h^{0,2}(S)=0$ so these are the usual GW/DT invariants.
\smallskip

\noindent \textbf{Acknowledgements.} We would like to thank Daniel Huybrechts, Vivek Shende and Rahul Pandharipande for useful discussions. Both authors were supported by EPSRC programme grant number EP/G06170X/1.

\section{The moduli space as a zero locus}

We fix some notation. Let $S$ be a nonsingular projective surface with cohomology class $\beta \in H^{2}(S,\Z)$. In this paper $L$ always denotes a line bundle with $c_1(L)=\beta$ and $h$ is the arithmetic genus (\ref{genus}) of curves in class $\beta$.

%. Let $h$ denote the arithmetic genus of curves in class $\beta$, as given by the adjunction formula
%$$
%2h-2=\beta^2-c_1(S).\beta.
%$$

For stable pairs on $S$ in class $\beta$ (or on $X:=K_S$ in class $\iota_*\beta$) with holomorphic Euler characteristic $$\chi=1-h+n$$ we refer to \cite{KT1, PT3}. In this paper we need only their description as pairs $(C,Z)$ where $C\subset S$ is a pure curve in class $\beta$ and $Z\subset C$ is a length-$n$ subscheme. This extends to give a set-theoretic isomorphism of moduli spaces
\beq{isom}
P_\chi(S,\beta)\ \cong\ \Hilb^n(\mathcal C/H_\beta).
\eeq
Here $\mathcal C\to H_\beta$ is the universal curve over the Hilbert scheme $H_\beta:=\Hilb_\beta(S)$ of pure curves in class $\beta$, and $\Hilb^n$ is the relative Hilbert scheme of $n$ points on the fibres of $\mathcal C$. In \cite[Appendix B]{PT3} it is shown that \eqref{isom} is an isomorphism of schemes. And 
in \cite[Appendix A]{KT1}, written with D. Panov, it is more-or-less
shown\footnote{What is shown is that the two tangent-obstruction complexes are the same, but it is not checked that the maps to the cotangent complex agree. This is not important for producing a virtual cycle, which only depends on the K-theory class of the obstruction complexes.} that \eqref{isom} is an isomorphism of \emph{schemes with perfect obstruction theory}. Here we have to take the \emph{reduced} obstruction theory on the left hand side, and on the right hand side the obstruction theory arising from a natural description of the relative Hilbert scheme in terms of equations. We give a brief account of this description in 2 steps now; for full details see \cite{KT1}. \medskip

$\bullet$ Pick a divisor $A\subset S$, sufficiently positive that $L(A)$ is very ample with no higher cohomology for all $L\in\Pic_\beta(S)$. Then by adding $A$ to divisors we get an embedding of $\Hilb_\beta(S)$ into $\Hilb_\gamma(S)$, where $\gamma=[A]+\beta$:
$$
\xymatrix{H_\beta\ \ar@{^(->}[r]^(.47){+A} & H_\gamma.}
$$
Now $H_\gamma$ is smooth (it is a projective bundle over $\Pic_\gamma(S)$) and the image $A+H_\beta$ is the set of divisors $D\in H_\gamma$ which contain $A$, i.e. the divisors $D$ for which
\beq{zeroDA}
s_D|_A=0\,\in H^0(\O(D)|_A),
\eeq
where $s_D\in H^0(\O(D))$ is the equation defining $D$. Varying \eqref{zeroDA} over $H_\gamma$ we get a section of a bundle 
\beq{sF}
s_{\mathcal D}|_{H_\gamma\times A}\ \ \mathrm{of}\ \ F:=\pi_{\gamma*}\big(\O(\mathcal D)|_{H_\gamma\times A}\big)
\eeq
over $H_\gamma$, whose zero locus is precisely $H_\beta$. Here $\mathcal D\subset S\times H_\gamma$ is the universal divisor, and $\pi_S,\,\pi_\gamma$ are the projections from $S\times H_\gamma$ to its two factors. The above pushdown has no higher cohomology due to Condition \eqref{h2=0}. This description of $H_\beta$ in terms of equations endows it with a natural perfect obstruction theory. \medskip

$\bullet$ Secondly, we embed
$$
\xymatrix{\Hilb^n(\mathcal C/H_\beta)\ \ar@{^(->}[r] & S^{[n]}\times H_\beta,}
$$
where $S^{[n]}$ denotes the (smooth) Hilbert scheme of $n$ points on $S$. A point $(Z,C)$ of $S^{[n]}\times H_\beta$ is in the image if and only if $Z\subset C$, if and only if
\beq{zeroZC}
s_C|_Z=0\,\in H^0(\O_Z(C)).
\eeq
Varying \eqref{zeroZC} over $S^{[n]}\times H_\beta$ we get a section of a bundle
\beq{sCn}
s_{\mathcal C}|_{\mathcal Z \times H_\beta}\ \ \mathrm{of}\ \ \O(\mathcal C)^{[n]}:=\pi_*\big(\O(\mathcal C)|_{\mathcal Z \times H_\beta})\big),
\eeq
whose zero locus is precisely $\Hilb^n(\mathcal C/H_\beta)$. Here $\mathcal C\subset S\times H_\beta$ is the universal divisor, $\mathcal Z\subset S\times S^{[n]}$ is the universal length-$n$ subscheme of $S$, and $\pi$ is the projection $S\times S^{[n]}\times H_\beta\to S^{[n]}\times H_\beta$. This description of $\Hilb^n(\mathcal C/H_\beta)$ in terms of equations (relative to the possibly singular space $H_\beta$) endows it with a natural perfect relative obstruction theory over $H_\beta$. \medskip

In \cite[Appendix A]{KT1} we show how to combine these two obstruction theories to endow $\Hilb^n(\mathcal C/H_\beta)$ with a perfect absolute obstruction theory, which we then identify with the reduced obstruction theory of stable pairs.
Noting that over $S\times H_\beta$ the line bundle\footnote{We suppress many pullbacks for readability; here $A$ denotes $\pi_S^*A$.}\, $\O(\mathcal D-A)$ restricts to $\O(\mathcal C)$, we see that the bundle $\O(\mathcal C)^{[n]}$ of \eqref{sCn} extends naturally over $S^{[n]}\times H_\gamma$ as $\O(\mathcal D - A)^{[n]}$. (Its section $s_{\mathcal C}|_{\mathcal Z \times H_\beta}$ does not extend.) As a consequence we get the following.

\begin{theorem} {\cite[Theorem~A.7]{KT1}} \label{main}
Assuming Condition (\ref{h2=0}), the pushforward of the reduced virtual cycle
$$
[P_{1-h+n}(S,\beta)]^{red} \in H_{2v}(P_{1-h+n}(S,\beta))
$$
to the smooth ambient space $S^{[n]} \times H_{\gamma}$ is Poincar\'e dual to
\begin{equation*}
c_{r}(F)\,.\,c_{n}\big(\O(\D - A)^{[n]}\big).
\end{equation*}
\end{theorem}

\noindent
Here $v = h-1+n + \int_{\beta} c_{1}(S) + h^{0,2}(S)$ is the reduced virtual dimension of $P_{1-h+n}(S,\beta)$ and $r=\chi(L(A))-\chi(L)$ is the rank of the bundle $F$ of \eqref{sF}. Also, given any family $\mathcal L \to S \times B$ of line bundles on $S$, we use the notation $\mathcal L^{[n]}$ for the rank $n$ vector bundle on $S^{[n]} \times B$ defined by pulling $\mathcal L$ back to $\mathcal Z \times B$ and pushing forward to $S^{[n]} \times B$.

\subsection{Virtual normal bundle}

Integrating insertions against the reduced virtual class of Theorem \ref{main} gives numerical invariants of $S$ as in \cite{KT1}. These are part of the full residue invariants of \cite{KT1}, but to get all of them we must include the term\footnote{Writing $N^{vir}$ as a two-term complex $E_0\to E_1$ of equivariant bundles \emph{whose weights are all nonzero} (which is possible, and ensures that the $c_{top}(E_i)$ are invertible), $e(N^{vir})$ is defined to be
$c_{top}(E_0)/c_{top}(E_1)$, where $c_{top}$ is the top $\C^*$-equivariant Chern class.}
$$
\frac1{e(N^{vir})}\in H^*_{\C^*}(P_{1-h+n}(S,\beta))\otimes_{\Z[t]}\Z(t)
\cong H^*(P_{1-h+n}(S,\beta))\otimes_\Z\Z(t),
$$
which arises in the virtual localisation formula of \cite{GP}. Here $N^{vir}$ is the virtual normal bundle of the inclusion $P_{1-h+n}(S,\beta)\subset P_{1-h+n}(X,\iota_*\beta)$, i.e. the dual of the moving part of the reduced 3-fold stable pairs obstruction theory.

Using Serre duality and the fact that $X=K_S$ has trivial canonical bundle $K_X\cong\O_X\otimes\mathfrak t^*$ with $\C^*$-action of weight $-1$, it turns out \cite[Proposition 3.4]{KT1} that $N^{vir}$ is the ordinary (i.e.~not reduced) deformation-obstruction complex $E\udot$ of $P_{1-h+n}(S,\beta)$, shifted by $[-1]$ and twisted by the $\C^*$-representation of weight $1$:
$$
N^{vir}=E\udot[-1]\otimes\mathfrak t=(R\pi_{P*}R\hom(\II\udot_S,\mathbb F))^\vee[-1]\otimes\mathfrak t.
$$
See \cite{KT1} for the meaning of this notation (though we will not need it here). In \cite[Proposition A.3]{KT1} the obstruction theory $E\udot$ was shown to sit in an obvious exact triangle with the  usual obstruction theory $(R\pi_{\beta*}\O_{\mathcal C}(\mathcal C))^\vee$ for $H_\beta$ and the relative obstruction theory
$$
\xymatrix@C=30pt{
\big\{\big(\O(\mathcal C)^{[n]}\big)^* \ \ar[r]^(.60){ds_{\mathcal C}|_{\mathcal Z \times H_\beta}} & \ \Omega_{S^{[n]}}\big\}}
$$
of $P_{1-h+n}(S,\beta)\big/H_\beta$ arising from the description \eqref{sCn}. We have again suppressed some pullback maps, and used $\pi_\beta$ to denote the projection $S\times H_\beta\to H_\beta$. Combining these facts shows that at the level of K-theory,
$$
[N^{vir}]=\big[\big(\O(\mathcal C)^{[n]}\big)^*-\Omega_{S^{[n]}}-(R\pi_{\beta*}\O_{\mathcal C}(\mathcal C))^\vee\big]\otimes\mathfrak t.
$$
This can be expressed as either
\beq{K1}
[N^{vir}]=\big[\big(\O(\mathcal C)^{[n]}\big)^*-\Omega_{S^{[n]}}-(R\pi_{\beta*}
\O(\mathcal C))^\vee+R\Gamma(\O_S)^\vee\otimes\O\big]\otimes\mathfrak t,
\eeq
or, using the exact sequence $0\to\O_{\mathcal C}(\mathcal C)\to\O_{\mathcal D}(\mathcal D)\to\O_A(\mathcal D)\to0$
(and the fact that $\mathcal D-A$ restricts to $\mathcal
C$ on $H_\beta\subset H_\gamma$),
\beq{K2}
\ [N^{vir}]=\big[\big(\O(\mathcal D-A)^{[n]}\big)^*-\Omega_{S^{[n]}}-
(\pi_{\gamma*}\O(\mathcal D))^\vee+R\Gamma(\O_S)^\vee\otimes\O+F^*\big]\otimes\mathfrak t.
\eeq
Recall that $F$ is the bundle \eqref{sF} and $\pi_\gamma$ is the projection $S\times H_\gamma\to H_\gamma$. In the form \eqref{K2} it is clear that $[N^{vir}]$ is the restriction of a class on the ambient space $S^{[n]}\times H_\gamma$, which will make calculation of the general invariants possible: see Section \ref{without}. To compute invariants of a single linear system we will find it convenient to use the form \eqref{K1} -- see Section \ref{H1} -- though of course we could also have used \eqref{K2}.

\section{Calculation with $H_1$-insertions} \label{H1}

In this Section we compute the reduced residue stable pair invariants
$$
\cP_{\!\chi,\beta}^{red}(S,[\gamma_{1}] \ldots [\gamma_{b_{1}(S)}][pt]^{m})
=\int_{[P_\chi(S,\beta)]^{red}\ }\frac1{e(N^{vir})}
\left(\prod_{i=1}^{b_1(S)}\tau(\gamma_i)\right)\!\tau([pt])^m
$$
lying in $\Z(t)$. Recall that $\chi=1-h+n$; otherwise we use the notation of \cite[Section 3.2]{KT1}. The $\gamma_i$ form an integral oriented basis of $H_1(S)$/torsion; their insertion cuts $\Hilb_\beta(S)$ down to a single linear system $|L|$. The $m$ point insertions further cut this down to a codimension-$m$ linear subsystem. In fact by \cite[Section 4]{KT1}, particularly Equations (52, 54), the above equals
\beq{simp}
\int_{j^![P_\chi(S,\beta)]^{red}}\ \frac{h^m}{e(N^{vir})}\,,
\eeq
where $j^!$ is the refined Gysin map \cite[Section 6.2]{Ful} for the Cartesian diagram
$$
\xymatrix@=20pt{P_\chi(S,|L|)\ \ar@{^(->}[r]\ar[d] & \,P_\chi(S,\beta) \ar[d] \\
\{L\}\ \ar@{^(->}[r]^(.43)j & \,\Pic_\beta(S),}
$$
and $h$ is the pullback of the hyperplane cohomology class from $|L|$ to $P_\chi(S,|L|) \break \cong \Hilb^n(\mathcal C/|L|)$. Factor this through the diagram
$$
\xymatrix@C=20pt@R=16pt{P_\chi(S,|L|)\ \ar@{^(->}[r]\ar@{^(->}[d]^{\iota\_L} & \,P_\chi(S,\beta) \ar@{^(->}[d]^\iota \\
S^{[n]}\times|L(A)|\ \ar@{^(->}[r]^(.47)\jmath\ar[d] & \,S^{[n]}\times\Hilb_\gamma(S) \ar[d] \\
\{L(A)\}\ \ar@{^(->}[r]\ar@{=}[d] & \,\Pic_\gamma(S) \ar@{=}[d] \\
\{L\}\ \ar@{^(->}[r]^(.48)j & \,\Pic_\beta(S).\!\!}
$$
The central vertical arrows are flat, so $j^!=\jmath^!$. Thus by \cite[Theorem 6.2]{Ful} we have
$$
\iota_{L*}j^![P_\chi(S,\beta)]^{red}= j^! \iota_* [P_\chi(S,\beta)]^{red}
=\jmath^*\iota_*[P_\chi(S,\beta)]^{red}.
$$
Therefore by Theorem \ref{main}, \eqref{simp} becomes
$$
\int_{S^{[n]}\times|L(A)|}
c_r(F)\,.\,c_n(\O(\mathcal D-A)^{[n]})\,\frac{h^m}{e(N^{vir})}\,,
$$
where as usual we have suppressed the pullback maps $\jmath^*$ on the bundles $F$ and $\O(\mathcal D-A)^{[n]}$.
Over $S\times|L(A)|$, the line bundle $\O(\mathcal D)$ is isomorphic to $L(A)\boxtimes\O(1)$ as both have a section cutting out $\mathcal D$. Hence
$$
F|_{S^{[n]}\times|L(A)|}\cong H^0(L(A)|_A)\otimes\O(1)\cong\O(1)^{\oplus r},
$$
where $r=\chi(L(A))-\chi(L)$, and our integral becomes
\beq{sofar}
\int_{S^{[n]}\times|L(A)|}h^rc_n(L^{[n]}(1))\,\frac{h^m}{e(N^{vir})}
\ =\ \int_{S^{[n]}\times\PP^{\,\chi(L)-1-m}}\frac{c_n(L^{[n]}(1))}{e(N^{vir})}\,.
\eeq

We use the following notation. For any bundle $E$ and variable $x$, set $c_{x}(E):= 1+c_{1}(E)x+c_{2}(E)x^{2}+\ldots$\ . Thus if $E$ has rank $r$ then 
\beq{chernt}
e(E \otimes \mathfrak{t}) = \sum_{i=0}^{r} t^{i} c_{r-i}(E) = t^{r} \sum_{i=0}^{r} (-1/t)^{r-i}(-1)^{r-i} c_{r-i}(E) = t^{r} c_{-1/t}(E^{*}),
\eeq
where $t:=c_1(\mathfrak t)$ is the equivariant parameter: the generator of $H^*(B\C^*)$. Use this to substitute the expression \eqref{K1} for $[N^{vir}]$ into \eqref{sofar}. Since $R\pi_{\beta*}\O(\mathcal C)=R\Gamma(L)\otimes\O(1)$, we get
$$
t^{2n+\chi(L)-n-\chi(\O_S)}\int_{S^{[n]}\times
\PP^{\,\chi(L)-1-m}\ }c_n(L^{[n]}(1))\,\frac{c_{-1/t}\big(T_{S^{[n]}}\big)c_{-1/t}\big(\O(1)^{\oplus\chi(L)}\big)}
{c_{-1/t}\big(L^{[n]}(1)\big)}\,.
$$
Since only the degree $n+\chi(L)-1-m$ part of the quotient contributes to the integral we get
$$
t^{n+\chi(L)-\chi(\O_S)}\!\left(\!-\frac1t\right)^{\!n+\chi(L)-1-m}\!\!\!
\int_{S^{[n]}\times\PP^{\,\chi(L)-1-m}}c_n(L^{[n]}(1))\,\frac{c_\bull\big(T_{S^{[n]}}\big)
c_\bull\big(\O(1)^{\oplus\chi(L)}\big)}{c_\bull\big(L^{[n]}(1)\big)}\,,
$$
where $c_\bull$ denotes the total Chern class.
Integrating over $\PP^{\,\chi(L)-1-m}$ leaves
\begin{equation*} 
(-1)^{\chi(L)-1-m+n}t^{m+1-\chi(\O_{S})}\!\!\int_{S^{[n]}}\left[\frac{c_{\bull}
(T_{S^{[n]}})(1+h)^{\chi(L)} \sum_{i=0}^{n} h^{i} c_{n-i}(L^{[n]})}{\sum_{i=0}^{n} (1+h)^{i} c_{n-i}(L^{[n]})}\right]_{h^{\chi(L)-1-m}}
\end{equation*}
where the suffix means we take the coefficient of $h^{\chi(L)-1-m}$ in the bracketed expression.

The right hand side is a tautological integral over $S^{[n]}$, involving only Chern classes of $L^{[n]}$ and the tangent bundle.
Applying the recursion of [EGL] $n$ times, it becomes an integral over $S^n$ of a polynomial in $c_1(L),\,c_1(S),\,c_2(S)$ (pulled back from different $S$ factors) and $\Delta_*1,\,\Delta_*c_1(S),\,\Delta_*c_1(S)^2,\,\Delta_*c_2(S)$ (pulled back from different $S\times S$ factors),
where $\Delta\colon S\into S\times S$ is the diagonal. The result is a degree $n$ universal polynomial in $c_{1}(L)^{2}$, $c_{1}(L).c_{1}(S)$, $c_{1}(S)^{2}$ and $c_{2}(S)$ (see also \cite[Section~4]{KST}). This proves Theorem \ref{1}.

\section{Calculation without $H_1$-insertions} \label{without}

Now we turn to the calculation of the reduced residue stable pair invariants
\beq{aim}
\cP_{\!\chi,\beta}^{red}(S,[pt]^{m})\ =\ \int_{[P_\chi(S,\beta)]^{red}\ }
\frac1{e(N^{vir})}\tau([pt])^m\ \in\ \Z(t).
\eeq
When $b_1(S)>0$ this differs from the invariant calculated in Section \ref{H1} as we have to integrate also over $\Pic(S)$.

Picking a Poincar\'e bundle $\Pg$ over $S\times\Pic_\gamma(S)$ expresses $\Hilb_\gamma(S)$ as a projective bundle over $\Pic_\gamma(S)$:
\beq{projb}
\Hilb_\gamma(S)=\PP(p_*\Pg)\Rt{AJ}\Pic_\gamma(S),
\eeq
where $p$ is the projection $S\times\Pic_\gamma(S)\to\Pic_\gamma(S)$ and $AJ$ is the Abel-Jacobi map. Fix a point $x\in S$. Then the locus of curves in $\Hilb_\gamma(S)$ passing through $x$,
$$
D_x:=\PP\big(p_*(\Pg\otimes\I_{\{x\}\times\Pic_\gamma(S)})\big)\subset\Hilb_\gamma(S),
$$
is a divisor since it defines a hyperplane in each projective space fibre (by the very ampleness of the class $\gamma$).

Of course $\Pg$ is only unique up to tensoring by line bundles pulled back from $\Pic_\gamma(S)$. By choosing that line bundle to be
$\Pg^{-1}|_{\{x\}\times\Pic_\gamma(S)}$ if necessary, we may assume without loss of generality that $\Pg$ is trivial at $x$:
\beq{normal}
\Pg|_{\{x\}\times\Pic_\gamma(S)}\cong\O_{\Pic_\gamma(S)}.
\eeq

\begin{lemma} Under the normalisation \eqref{normal}, the hyperplane line bundle $\O(1)$ of the projective bundle \eqref{projb} is $\O(D_x)$.
\end{lemma}

\begin{proof}
Using the normalisation \eqref{normal},
the tautological bundle $\O(-1)\into p_*\Pg$ of $\PP(p_*\Pg)$ has a canonical map to $\O$ given by evaluation of sections at $x\in S$. Its zero locus is precisely $D_x$.
\end{proof}

\begin{corollary}
The insertion $\tau([pt])$ is the cohomology class $h:=c_1(\O(1))$ pulled back to $P_\chi(S,\beta)$ via $P_\chi(S,\beta)\subset S^{[n]}\times H_\gamma\to H_\gamma$.
\end{corollary}

\begin{proof}
Recall \cite[Section 3.2]{KT1} that we use $[\ \cdot\ ]$ to denote Poincar\'e duals, and that $\tau([pt])\in H^*(P_\chi(S,\beta))$ is defined by the top half of the diagram
$$
\xymatrix{S \\ P_\chi(S,\beta)\times S \ar[u]^{\pi_S}\ar[d]\ar[r]^(.55){\pi_P} & P_\chi(S,\beta) \ar[d] \\
H_\beta\times S \ar[r]^(.55){\pi_\beta} & H_\beta.\!\!}
$$
Namely $\tau([pt])=\pi_{P*}(\pi_S^*[pt]\cdot c_1(\mathbb F))$ on $P_\chi(S,\beta)$, where $\mathbb F$ is the universal sheaf over $S\times P_\chi(S,\beta)$. But $c_1(\mathbb F)$ is the pullback of $c_1(\O(\mathcal C))$ from $H_\beta\times S$. So by going round the above Cartesian square, we find that $\tau([pt])$ is the pullback from $H_\beta$ of $\pi_{\beta*}(c_1(\O(\mathcal C))\boxtimes[pt])$.

Now the terms in the above square embed (via the obvious commuting maps) in the terms in the square
$$
\xymatrix{S^{[n]}\times H_\gamma\times S \ar[d]\ar[r] & S^{[n]}\times H_\gamma \ar[d] \\
H_\gamma\times S \ar[r]^(.55){\pi_\gamma} & H_\gamma.\!}
$$
Our class $\pi_{\beta*}(c_1(\O(\mathcal C))\boxtimes[pt])$ is the restriction to $H_\beta$ of $\pi_{\gamma*}(c_1(\O(\mathcal D-A))\boxtimes[pt])$. Since the class $A$ is pulled back from $S$ it contributes nothing for degree reasons. And since $H_\gamma\times S$ is smooth we can use Poincar\'e duality to write the rest as
the pushdown via $\pi_\gamma$ of the homology class of $\mathcal D$ intersected with that of $H_\gamma\times\{x\}$. This intersection is $D_x$ and is transverse. By the Lemma we therefore get $h\in H^2(H_\gamma)$. Pulling up to $S^{[n]}\times H_\gamma$ and restricting to $P_\chi(S,\beta)$ in the previous square gives the result.
\end{proof}

Substituting this result and the expression \eqref{K2} for $e(N^{vir})$ into \eqref{aim} gives the expression
$$
t^{2n+\chi(L(A))-n-\chi(\O_S)-r}\hspace{-5mm}\mathop{\int}_{S^{[n]}\times H_\gamma}\!\!\!\!\!
c_r(F)c_n\big(\O(\mathcal D-A)^{[n]}\big)h^m
\frac{c_{-1/t}(T_{S^{[n]}})c_{-1/t}\big(\pi_{\gamma*}\O(\mathcal D))\big)}
{c_{-1/t}\big(\O(\mathcal D-A)^{[n]}\big)c_{-1/t}(F)}\,,
$$
by using the identity \eqref{chernt}. Recall that $r:=\rk(F)=\chi(L(A))-\chi(L)$. The piece of the quotient in the correct degree to contribute is
$$
\left(-\frac1t\right)^{n+\chi(L)-1+h^{0,1}(S)-m}\frac{c_\bull(T_{S^{[n]}})c_\bull
\big(\pi_{\gamma*}\O(\mathcal D))\big)}
{c_\bull\big(\O(\mathcal D-A)^{[n]}\big)c_\bull(F)}\,.
$$
Thus we are left with the product of $(-1)^{\chi(L)-1+n-m+h^{0,1}(S)}t^{m-h^{0,2}(S)}$ and
\beq{padd}
\int_{S^{[n]}\times H_\gamma}c_r(F)c_n\big(\O(\mathcal D-A)^{[n]}\big)\,h^m\,
\frac{c_\bull(T_{S^{[n]}})c_\bull\big(\pi_{\gamma*}\O(\mathcal D))\big)}
{c_\bull\big(\O(\mathcal D-A)^{[n]}\big)c_\bull(F)}\,.
\eeq
This takes care of the sign and power of $t$ in Theorem \ref{2}. We will now concentrate on the integral \eqref{padd}, first pushing it down the projective bundle \eqref{projb} to $S^{[n]}\times\Pic(S)$, then to $\Pic(S)$, then finally to a point. 

\addtocontents{toc}{\SkipTocEntry}
\subsection*{Integrating over the fibres of the Abel-Jacobi map}
Since the line bundle $\hom(\O(-1),\Pg)$ has a canonical section cutting out $\mathcal D$, we have the identity
$$
\Pg(1)\cong\O(\mathcal D).
$$
Substituting into \eqref{padd} yields
$$
\int_{S^{[n]}\times H_\gamma}c_r(F)c_n\big(\Pg(-A)^{[n]}(1)\big)\,h^m\,
\frac{c_\bull(T_{S^{[n]}})c_\bull\big(\big[Rp_*\Pg(-A)\big](1)\big)}
{c_\bull\big(\Pg(-A)^{[n]}(1)\big)}\,.
$$
Expand the integrand in powers of $h=c_1(\O(1))$. Notice that everything is now pulled back from $S^{[n]}\times\Pic_\gamma(S)$ except $c_r(F)$ and the powers of $h$. These are dealt with by the following Lemma for pushing down the projective bundle \eqref{projb}. We use the following diagram
\beq{mapss}
\xymatrix@R=20pt{
S\times H_\gamma \ \ar[r]^(.4){1_{S} \times AJ}\ar[d]_{\pi_\gamma} & \ S\times\Pic_\gamma(S) \ar[d]^p \\ H_\gamma \ar^(.4){AJ}[r] & \Pic_\gamma(S),\!}
\eeq
%and use the same notation when we take the product of everything with $S^{[n]}$.
Recall that the Segre classes $s_i\in H^{2i}$ are defined by $s_\bull=1/c_\bull$\,.

\begin{lemma}
The pushdown $AJ_*(c_r(F)h^j)$ to $\Pic_\gamma(S)$ is equal to the Segre class $s_{j-\chi(L)+1}\big(Rp_*\Pg(-A)\big)$.
\end{lemma}

\begin{proof} 
Using $\Pg(1) \cong \O(\mathcal D)$ and diagram (\ref{mapss})
$$
F=\pi_{\gamma*}\big(\O(\mathcal D)|_{H_\gamma\times A}\big)\cong AJ^* p_*\big( \Pg|_{\Pic_\gamma(S) \times A}\big)(1),
$$
and so
$$
c_r(F)= AJ^* \sum_{i=0}^rc_{r-i}\big(p_*\big(\Pg|_{\Pic_\gamma(S)\times A}\big)\big)h^i.
$$
We can push down using the standard identity \cite[Section 3.1]{Ful}
\beq{Segr}
AJ_*(h^i)=s_{i-\chi(L(A))+1}(p_*\Pg).
\eeq
We get
\beqa
AJ_*(c_r(F)h^j) &=& \sum_{i=0}^rc_{r-i}\big(p_*\big(\Pg|_{\Pic_\gamma(S)\times A}\big)\big)s_{i+j-\chi(L(A))+1}(p_*\Pg) \\
&=& \big[c_\bull\big(p_*\big(\Pg|_{\Pic_\gamma(S)\times A}\big)\big)s_\bull(p_*\Pg)
\big]_{r+j-\chi(L(A))+1} \\
&=& s_{j-\chi(L)+1}\big(p_*\Pg-p_*\big(\Pg|_{\Pic_\gamma(S)\times A}\big)\big) \\ &=& s_{j-\chi(L)+1}\big(Rp_*\Pg(-A)\big). 
\eeqa
\vspace{-1.3cm} \[ \qedhere \]
\end{proof}

\noindent\textbf{Remark}. \emph{
Note that by integrating out $c_r(F)$ we are passing from $H_\gamma$ back to (the reduced virtual cycle of) $H_\beta$. In the case where $\beta$ is sufficiently ample that no virtual technology is necessary, we could have worked directly on $\PP(p_{*}\Pb)$ and pushed down $h^j$ with no $c_r(F)$ insertion. As in \eqref{Segr} this would have given $s_{j-\chi(L)+1}(p_*\Pb)$, the same result as in the Lemma.} \\

Thus our integral has become one over $S^{[n]}\times\Pic_\gamma(S)$ of a polynomial $Q$ in the Chern classes of $T_{S^{[n]}},\ Rp_*\Pg(-A)$ and $\Pg(-A)^{[n]}$.
Since $$\Pb:=\Pg(-A)$$ is a Poincar\'e bundle for $\Pic_\beta(S)$, we use $\otimes\O(-A)$ to identify $\Pic_\gamma(S)$ with $\Pic_\beta(S)$ to get an integral
\beq{cal}
\int_{S^{[n]}\times\Pic_\beta(S)}Q\left(c_\bull(T_{S^{[n]}}),c_\bull(Rp_*\Pb),
c_\bull\big(\Pb^{[n]}\big)\right).
\eeq
By this notation we mean that $Q$ is a polynomial in all of the components $c_i$ of $c_\bull$ (rather than just in the total Chern classes themselves). Notice this expression is now manifestly independent of $A$.

\addtocontents{toc}{\SkipTocEntry}
\subsection*{Integrating over the Hilbert scheme of points}
We need a family version of the recursion of \cite{EGL}.

We fix an arbitrary base $B$, a line bundle $\L$ on $S\times B$ a cohomology class of the form
\begin{equation*}
P(c_\bull(\L^{[n]}), c_\bull(T_{S^{[n]}})) \quad\mathrm{on}\ S^{[n]}\times B,
\end{equation*}
for some polynomial $P$. We wish to push it down to $B$. The recursion \cite{EGL} is easily checked to apply (though it was actually written for the case $B=pt$). The pushdown is turned first into one down $S^{[n-1]}\times S\times B$, then $S^{[n-2]}\times S^2\times B$, and so on. The end result is an integral down the fibres of $S^n\times B\to B$ of a polynomial in
\begin{itemize}
\item $c_1(\L),\,c_1(S),\,c_2(S)$ pulled back from different $S\times B$ and $S$ factors,
\item $\Delta_*1,\,\Delta_*c_1(S),\,\Delta_*c_1(S)^2,\,\Delta_*c_2(S)$ pulled back from different $S\times S$ factors,
where $\Delta$ is the diagonal $S\into S\times S$.
\end{itemize}
In turn this integral is easily computed as a polynomial in integrals down $p\colon S\times B\to B$ of products of
$c_1(S),\,c_2(S),\,c_1(\L),\,p^*p_*\big(c_1(\L)^i c_j(S)^k\big)$.
Applied to $B=\Pic_\beta(S)$ and $\L=\Pb$, \eqref{cal} becomes a polynomial in terms
$$
\int_{S\times\Pic_\beta(S)}M\Big(c_1(S),c_2(S),c_\bull(Rp_*\Pb),
c_1(\Pb),p^*p_*\big(c_1(\Pb)^i c_j(S)^k\big)\Big),
$$
where $M$ is any monomial and $j,k=0,1,2$.

\addtocontents{toc}{\SkipTocEntry}
\subsection*{Integrating over the Picard variety}
Next we apply Grothendieck-Rie\-mann-Roch,
$$
ch(Rp_*\Pb)=p_*\big[\exp(c_1(\Pb)) \Td(S)\big],
$$
and the decomposition
$$
c_1(\Pb)\ =\ (\beta,\id,0)\ \in\
H^2(S)\,\oplus\,\big(H^1(S)\otimes H^1(\Pic_\beta(S))\big)\,\oplus\,H^2(\Pic_\beta(S)).
$$
Here we use the canonical identification $H^1(\Pic_\beta(S))\cong H^1(S)^*$ (so that $\id\in\mathrm{End}\,H^1(S)$) and the normalisation \eqref{normal}. The upshot is a polynomial in the integrals
$$
\int_{S\times\Pic_\beta(S)}M\Big(c_1(S),c_2(S),\beta,\id,
p^*p_*\big(\beta^i.\id^j\!.\,c_k(S)^l\big)\Big),
$$
for arbitrary $i,j$ and $k,l=0,1,2$. For degree reasons, pushing down to $S$ we get a polynomial in the terms $\beta^2,\,\beta.c_1(S),\,c_1(S)^2,\,c_2(S)$ and
\beq{terms}
\int_{\Pic_\beta(S)}M\big(p_*(\id^4),p_*(\beta.\id^2),p_*(c_1(S).\id^2)\big).
\eeq
Using the identification $\Lambda^{2}H^{1}(S,\R)^{*} \cong H^{2}(\Pic_{\beta}(S),\R)$ we obtain
\beqa
p_*(\beta.\id^2) &=& -2[\beta], \\
p_*(c_{1}(S).\id^2) &=& -2[c_{1}(S)], \\
p_*(\id^4) &=& 24 [1].
\eeqa
where $[\beta]$, $[c_{1}(S)]$, $[1]$ are the classes defined in \eqref{pfa} in the Introduction. (To get the precise coefficients it is perhaps easiest to express everything in terms of a basis for $H^1(S)$ and its dual basis for $H^1(S)^*\cong H^1(\Pic_\beta(S))$ and then do the calculation.) 

Finally the canonical identification $\Lambda^{b_1(S)}H^1(S,\Z)^*\cong\Z$ given by wedging together an oriented integral basis is the same as the identification given by integrating over $\Pic_\beta(S)$. Therefore the integrals in \eqref{terms} are the numbers
$$
\Lambda^{i} [\beta] \wedge \Lambda^{j} [c_{1}(S)] \wedge \Lambda^{k} [1],\quad 2i+2j+4k = b_{1}(S),
$$
of \eqref{numbs}. This proves Theorem \ref{2}. \bigskip

\noindent\textbf{Remarks}.
The invariants $\Lambda^{i} [\beta] \wedge \Lambda^{j} [c_{1}(S)]  \wedge \Lambda^{k} [1]$ are in general distinct from $\beta^{2}$, $\beta.c_{1}(S)$, $c_{1}(S)^{2}$, $c_{2}(S)$ as can be seen by the following example. Let $S = \Sigma_{g} \times \PP^{1}$, where $\Sigma_{g}$ is a smooth projective curve of genus $g$. Under the identification
$$
H^2(S,\Z)\ \cong\ H^2(\Sigma_g)\,\oplus\,H^2(\PP^1)\ \cong\ \Z\oplus\Z,
$$
we write $\beta=(\beta_1,\beta_2)$. Then 
$$
\beta^{2} = 2\beta_{1} \beta_{2}, \ \beta.c_{1}(S) = 2\beta_{1} + (2-2g)\beta_{2}, \ c_{1}(S)^{2} = 8-8g,  \ c_{2}(S) = 4-4g.
$$
On the other hand, using the usual basis of $a$- and $b$-cycles for $H^1(S)^*\cong H_1(\Sigma_g)$ one computes that $[1]=0$ and
$$
\Lambda^i[\beta]\wedge\Lambda^{g-i}[c_1(S)] = 2^{g-i}g!\,\beta_{2}^i.
$$
\smallskip

For $S$ an abelian surface, however, the invariants $\Lambda^{i} [\beta] \wedge \Lambda^{j} [c_{1}(S)]  \wedge \Lambda^{k} [1]$ can all be expressed in terms of $\beta^{2}$, $\beta.c_{1}(S)$, $c_{1}(S)^{2}$, $c_{2}(S)$ (i.e. just $\beta^2$ since the others vanish). Indeed using the standard basis arising from a homeomorphism to $(S^{1})^{4}$ it is easy to see that $[1]=1$, $[c_1(S)]=0$ and
$$
\Lambda^2[\beta] = \int_{S} \beta^{2}.
$$

\section{Relation to 3-fold invariants} \label{3f}

In this section we discuss some cases in which the reduced residue stable pair invariants of $S$ computed in Theorem \ref{1} and Theorem \ref{2} are equal to the reduced residue stable pair invariants of the 3-fold $X=K_S$. Let $\iota\colon S\into X$ denote the inclusion.

First we give cases where  the fixed point locus $P_\chi(X, \iota_{*}\beta)^{\C^*}$ has no components other than $P_\chi(S,\beta)$. 
\begin{proposition} \label{irreducible}
There is an isomorphism $P_\chi(X,\iota_{*}\beta)^{\C^*} \cong P_\chi(S,\beta)$ if either
\begin{itemize}
\item $\beta$ is irreducible, or
\item $K_S^{-1}$ is nef, $\beta$ is $(2\delta+1)$-very ample
%\footnote{By this we mean there exists a line bundle in $\Pic_\beta(S)$ which is $(2\delta+1)$-very ample. Recall \cite{BS} that this means that $H^0(L)\to H^0(L|_Z)$ is surjective for every length $2\delta+2$ subscheme $Z$ of $S$.}
and $\chi \leq 1 - h + \delta$.
\end{itemize}
(As usual $h$ is defined by $2h-2=\beta^2-c_1(S).\beta$. The inequality on $n$ means the stable pairs have $\le\delta$ free points.)
\end{proposition}
\begin{proof}
Let $(F,s)\in P_\chi(X,\iota_{*}\beta)^{\C^*}$ be a $\C^*$-fixed stable pair in class $\beta$ with scheme-theoretic support $C_F$. Then its set-theoretic support $C_F^{red}$ lies in $S$. But $C_F$ has no embedded points (by the purity of $F$ \cite[Lemma 1.6]{PT1}) so if $\beta$ is irreducible then $C_F$ is in fact reduced. Thus $C_F$ and $(F,s)$ are pushed forward from $S$. So $(F,s)\in P_\chi(S,\beta)$ and $P_\chi(S,\beta)\into P_\chi(X,\iota_{*}\beta)^{\C^*}$ is a bijection. That they have the same scheme structure is proved in
\cite[Proposition~3.4]{KT1}. \medskip

Next assume instead that $\beta$ is $(2\delta+1)$-very ample, that $K_S^{-1}$ is nef, and that $F$ is not supported scheme-theoretically on $S$. We will show that $\chi(\O_{C_F})>1-h+\delta$, which implies the Proposition since $\chi(F)\ge\chi(\O_{C_F})$.

Let the irreducible components of $C_F$ be $C_{F,i}$, with underlying reduced varieties $C_i\subset S$. Since the $C_{F,i}$ are $\C^*$-fixed without embedded points, there is a sequence of integers $n_{i0}\ge\ldots\ge n_{ir_{i}}>0$ such that
\beq{grdd}
\O_{C_{F,i}}=\bigoplus_{k=0}^{r_{i}}\O_{n_{ik}C_i}\otimes K_S^{-k}
\eeq
as a graded ring. Here $n_{ik}C_i\subset S$ is the obvious divisor, and we are writing $\O_X$ as $\bigoplus_{k=0}^\infty K_S^{-k}$ by pushing down to $S$; the $\C^*$-action on $K_S$ then induces the obvious grading by $k$.
% We obtain 
% $$
% \beta = [C_{F,i}] = \sum_{k=0}^{r} n_{ik} [C_{i}].
% $$
Since $K_S^{-1}$ is nef, we obtain
\begin{align*}
\chi(\O_{C_{F,i}}) = \sum_{k=0}^{r_{i}}\big(\chi(\O_{n_{ik}C_{i}})
-k n_{ik} C_i.K_S\big) &\ge \sum_{k=0}^{r_{i}} \chi(\O_{n_{ik}C_i}) \\
&= -\frac12\sum_{k=0}^{r_i}\left(n_{ik}^2C_i^2 + n_{ik} C_i.K_S\right).
\end{align*}
In turn,
\begin{equation*}
\chi(\O_{C_{F}}) = \sum_i\chi(\O_{C_{F,i}})-\sum_{i<j}(n_{i0}n_{j0}+n_{i1} n_{j1} + \ldots)C_i.C_j.
\end{equation*}
Combining the two then adding the adjunction formula
$$
2h-2\ =\ \beta^2+K_S.\beta=\Big(\sum_{i,k} n_{ik} C_i\Big)^2+
\sum_{i,k}n_{ik}C_i.K_S
$$
yields
\begin{align} \nonumber
2\big(\chi(\O_{C_{F}})+h-1\big) &\geq \Big(\sum_{i,k}n_{ik}C_i\Big)^{2}-
\sum_{i,k}n_{ik}^2C_i^2-2\sum_{i<j}\sum_kn_{ik}n_{jk}C_i.C_j \\
&=\sum_{k\ne l}\sum_{i,j}n_{ik}n_{jl}C_i.C_j.
\end{align}
Setting $\beta_k=\sum_in_{ik}[C_i]$ to be the class of the $k$th graded piece of \eqref{grdd}, so that $\beta=\sum_k\beta_k$, we write this as
$$
\sum_{k\ne l}\beta_k.\beta_l\ =\ \beta^2-\sum_k\beta_k^2.
$$
By the Hodge index theorem, $a^2\le(L.a)^2/L^2$ for any positive $L\in H^{1,1}(S)$ and arbitrary $a\in H^{1,1}(S)$. (Proof: $a-(L.a)L/L^2$ is orthogonal to $L$ so has square $\le0$.) Applying this to $L=\beta$ and $a=\beta_k$ gives
\begin{align*}
2\big(\chi(\O_{C_{F}})+h-1\big) &\geq \beta^2-\sum_k\frac{(\beta.\beta_k)^2}{\beta^2} \\ &=\sum_k(\beta.\beta_k)\left(1-\frac{(\beta.\beta_k)}{\beta^2}\right)
\\ &=\sum_k\frac{(\beta.\beta_k)(\beta.(\beta-\beta_k))}{\beta^2} \\
&=\sum_k\frac{(\beta.\beta_k)(\beta.(\beta-\beta_k))}{(\beta.\beta_k)+(\beta.(\beta-\beta_k))}
\\ &\ge\sum_k\frac12\min\big(\beta.\beta_k\,,\,\beta.(\beta-\beta_k)\big).
\end{align*}
Both $\beta_k$ and $\beta-\beta_k$ are effective, since $C_F$ is not supported inside $S$. And the sum contains at least $2$ terms. So it is $\ge\beta.D$, for some effective divisor $D$. So it is sufficient to prove that $\beta.D>2\delta$ for effective classes $D$; in turn it is sufficient to prove this for irreducible $D$.

Choose $2\delta+2$ smooth points on $D$. By the definition of $(2\delta+1)$-very ampleness, there is divisor in $S$ in the class of $\beta$ which passes through the first $2\delta+1$ points, but not the last one. Therefore the divisor does not contain $D$, and $\beta.D\ge2\delta+1$, as required.
\end{proof}

The previous proposition is false for arbitrary surfaces. For instance if $K_S=\O_S(C_0)$, then consider $\beta=nC_0$ and let $C$ be the $n$-fold thickening of $C_0$ along the fibres of $K_S$. This is $\C^*$-fixed with $\chi = 1 - h$, but not scheme-theoretically supported on $S$. However one can often make it true again by restricting to small linear subsystems in the space of curves. This follows from the following result proved and used in \cite{KST}.

\begin{proposition} [{\cite[Proposition~2.1]{KST}}] \label{vample}
If $L$ is a $\delta$-very ample line bundle on $S$ then the general $\delta$-dimensional linear system $\PP^\delta\subset|L|$ contains a finite number of $\delta$-nodal curves appearing with multiplicity 1. All other curves in $\PP^{\delta}$ are \emph{reduced} with geometric genus $\overline g>h-\delta$ where $h$ is the arithmetic genus of curves in $|L|$. \hfill$\square$
\end{proposition}

(One can also assume that the curves in $\PP^\delta$ are irreducible if $L$ is $(2\delta+1)$-very ample, by \cite[Proposition 5.1]{KT1}.)

So when we deal with invariants of $X$ for stable pairs with divisor
class\footnote{Given a stable pair $(F,s)$ on $X$ in class $\iota_*\beta$, its pushdown $q_*F$ to $S$ has a divisor class $\div(q_*F)\in\Hilb_\beta(S)$: see \cite[Section 4]{KT1}. This is basically the support with multiplicities. We use insertions to force this class to lie in $\PP^\delta$.} lying in
$\PP^\delta\subset|L|\subset\Hilb_\beta(S)$, any $\C^*$-fixed pure curve has set-theoretic support on $S$ which is a reduced irreducible curve of class $\beta$. This must therefore be its scheme theoretic support too, so we get a bijection between the cut down moduli spaces $P_\chi(S,\PP^\delta)\cong P_\chi(X,\PP^\delta)^{\C^*}$ which is a scheme-theoretic isomorphism by \cite[Proposition 3.4]{KT1}. \medskip

Therefore in the situation of Proposition \ref{irreducible}, or Proposition \ref{vample} with full $H_1$-insertions and $\chi(L)-1-\delta$ point
insertions\footnote{The $H_1$-insertions cut the divisor class down to $|L|$ and then the point insertions further cut down to $\PP^\delta$. We call these insertions, when the conditions of Proposition \ref{vample} hold, the ``G\"ottsche case". There are no virtual cycles involved in this case since the resulting moduli space is smooth of the correct dimension. And in \cite[Section 5]{KT1} we show the resulting invariants contain the Severi degrees counting nodal curves studied by G\"ottsche.}, the reduced residue stable pair invariants of $X$ and $S$ coincide. Thus these invariants of $X$ are purely topological and determined by the universal polynomials of Theorem \ref{1} and Theorem \ref{2}. By the MNOP conjecture, this determines the corresponding reduced GW and DT invariants of $X$, which should therefore also be topological. Note that in the toric case $h^{0,2}(S)=0$, so these are the usual GW/DT invariants.

\bigskip

\noindent {\tt{mkool@math.ubc.ca}} \\
\noindent {\tt{richard.thomas@imperial.ac.uk}} \\

% \noindent Department of Mathematics \\
% \noindent Imperial College London\\
% \noindent South Kensington Campus \\
% \noindent London SW7 2AZ \\
% \noindent United Kingdom 

\end{document}